\documentclass[11pt,fleqn]{article}

\usepackage{amsmath}
\usepackage{booktabs}
\usepackage{amssymb}
\usepackage{amsthm}
\usepackage{textcomp}
\usepackage{graphics}
\usepackage{graphicx}
\usepackage{subfigure}
\usepackage{color}
\usepackage{enumerate}
\usepackage[round]{natbib}
\usepackage{bbm}

\newtheorem{theorem}{Theorem}
\newtheorem{prop}{Proposition}
\newtheorem{rem}{Remark}
\newtheorem{proclaim}{Assumption}
\newcommand{\sign}{\operatorname{sign}}

\def\ind{\mathbbm{1}}
\def\Real{\mathbbm{R}}
\def\E{\mathrm{E}}
\def\V{\mathrm{V}}

\begin{document}
\title{A note on the empirical process of strongly dependent stable random  variables }
\author{Emanuele Taufer}

\author{\emph{Emanuele Taufer} \\ Department of Economics and Management, University of Trento \\ \textsc{emanuele.taufer@unitn.it}
 }
\date{\today}

\maketitle

\begin{abstract} 
This paper analyzes the limit properties of the empirical process of $\alpha$-stable random variables with long range dependence. The $\alpha$-stable random variables are constructed by non-linear transformations of bivariate sequences of strongly dependent gaussian processes. The approach followed allows an analysis of the empirical process by means of  expansions in terms of  bivariate Hermite polynomials for the full range $0<\alpha<2$. A weak uniform reduction principle is provided and it is shown that the limiting process is gaussian. The results of the paper different substantailly from those available for empirical processes obtained by stable moving averages with long memory. An application to goodness-of-fit testing is discussed.


\end{abstract}

\bibliographystyle{plainnat}
\baselineskip = 1.1\baselineskip

\section{Introduction}

Consider a sequence of random variables (rv) $X_1, \dots, X_n$, with common continuous cumulative distribution function (CDF) $F$, constituting a sample from a strictly stationary and ergodic time series $\{ X_i, i \in \mathbb{Z} \}$ where $\mathbb{Z} = \{0, \pm 1, \pm 2, \dots \}$. For $\ind \{A\}$ being the indicator function of the event $A$, let $F_n$ denote the empirical distribution function (EDF) of the sequence, i.e. $F_n(x)= \frac{1}{n} \sum_{i=1}^n \ind \{X_i\leq x\}$. It is well known that the empirical process (EP)
\begin{equation}\label{EP}
\sqrt{n}(F_n(x)-F(x))
\end{equation}
converges to a non-degenerate Gaussian process either in the case where $\{X_i\}$ is a sequence of $i.i.d.$ or weakly dependent rv. 

The behavior of the EP is quite different in the case of long range dependence  (LRD) where proper normalizing constants are of order $n^{D/2}$, $0<D<1$ and the weak limit, if it exists, is a degenerate process in $x$. 

This paper studies the weak limit of $(F_n(x)-F(x))$, properly normalized, when the sample is formed by a sequence of strongly dependent stable random variables with index of stability $0<\alpha<2$.

One of the mainstream approaches in the study of LRD processes is via expansions, by means of orthogonal polynomials, of non-linear functionals of Gaussian LRD processes. In the case discussed here, if $F$ denotes the CDF of a stable rv $X$ and $\Phi$ the CDF of a standard normal rv $Z$, one has $\ind\{X\leq x\} = \ind\{F^{-1} \circ \Phi(Z)\leq x\}= \ind \{Z\leq F \circ \Phi^{-1} (x)\}$; in this framework it is quite simple to provide an expansion of the indicator function in an appropriate $L_2$ space. Howeve, given that analytic expressions of $F^{-1}$, with a few exceptions, are not available, this approach may not be optimal if one, for simulation, validation and testing purposes, wishes to generate stable rv given a sequence of LRD gaussian rv.

In this paper an approach based on a bivariate expansion is proposed. This will allow to provide fast and reliable methods of stable rv generation starting form and LRD gaussian sequence and, at the same time provide an analytic framework for the analysis of the EP. Some key results in this respect are due to \citet{CMS76} and \citet{Weron96} as far as  stable rv are concerned. Specific papers considering the EP of non-linear transformation of LRD gaussian sequences discussing techniques relevant here  are those of \citet{DehTaq89}, \citet{CsoMie96} and \citet{LeoSak01}. We also refer the interested reader to the excellent reviews of \citet{DehPhi02} for a general discussion on EP techniques and \citet{KouSur02} for a specific analysis of the LRD case. Other relevant literature discussing  bivariate (and multivariate) expansion on non-linear functionals of LRD gaussian sequences and other bivariate expansions are \citet{Arc94}, \citet{LeoTau01}, \citet{LeoSakTau02}, \citet{LedTaq11}, \citet{LeoTau13}, \citet{LedTaq14}. 

Another mainstream approach in the study of LRD processes, which will not be discussed here,  is based on linear processes (or moving averages). In this line of study, specific papers devoted to the EP  are those of \citet{HoHsi96}, \citet{GirSur99} and \citet{KouSur01} which, in particular, consider the case of stable innovations with $1 < \alpha <2$ and where a non-gaussian weak limit is obtained. 

It is worth noting that the approach followed here provides a discussion of the full range $0<\alpha<2$, new to the literature, and provides a gaussian weak limit. These results show the essential different nature, when outside the gaussian case, of LRD moving average processes and LRD processes obtained by non-linear transformation of gaussian sequences. 



The results obtained can find applications in the analysis of statistical functionals based on the EP. Relevant and recent examples in the literature concern the analysis of goodness of fit tests, such as, e.g. \citet{JamTau06}, \citet{Tau09}, \citet{DehTaq13}, \citet{KouSur13}, \citet{Gho13}.

This paper is organized as follows:  Section \ref{Sec:Background} contains background arguments while in Section \ref{Sec:HR} the EP of stable rv is discussed. A final section presents applications and simulations to substantiate the theoretical findings.

\section{Background}\label{Sec:Background}

In this section, some needed key features of stable rv will be recalled and a bivariate expansion, in terms of Hermite polynomials, of the EP of LRD stable random variables will be provided.

In order to define exactly the sequence $X$ of stable rv we state the following assumption where the classical set-up for a sequence of LRD gaussian random variables is defined:

\begin{proclaim}\label{A:1} Let $Z_i^{(1)}$ and $Z_i^{(2)}$ be independent copies of a sequence of gaussian random variables with null mean and unit variance and, for $j=1,2$,  $r(k)=\E(Z_i^{(j)},Z_{i+k}^{(j)}) = L(k) k^{-D}$ with $L(k)$ a slowly varying function and $0<D<1$.
\end{proclaim}

\subsection{Stable rv}\label{sub:Stablerv}

For $0 < \alpha \leq 2$, write $X \sim S_\alpha(\beta, \sigma, \mu)$ to denote an $\alpha$-stable rv with asymmetry $\beta \in[-1,1]$, scale $\sigma>0$ and location $\mu\in \mathbb{R}$, with characteristic function $\psi$ given by (here $i=\sqrt{-1}$)
\begin{equation}\label{Srep1}
\log \psi(z)= \begin{cases}
i \mu z -\sigma^\alpha |z|^\alpha [ 1- {i} \beta \sign(z) \tan (\frac{\pi \alpha}{2})], \quad \alpha \neq 1 \\
i \mu z -\sigma |z| [ 1+ {i} \beta \sign(z)\frac{ 2}{\pi} \log (|z|)],  \quad \alpha = 1.
\end{cases}
\end{equation}
An alternative representation, justified by considerations of analytic nature (see \citet{Zol86}, Theorem C.3), which will be relevant for our development is
\begin{equation}\label{Srep2}
\log\psi(z)= \begin{cases}
i \mu z -\sigma_2^\alpha |z|^\alpha \exp\{ - i \beta_2 \sign(z) \frac{\pi}{2} K(\alpha)\}, \quad \alpha \neq 1 \\
i \mu z -\sigma_2 |z| [ \frac{\pi}{2}+ i \beta_2 \sign(z) \log (|z|)],  \quad \alpha = 1
\end{cases}
\end{equation}
where $K(\alpha)= \alpha-1+ \sign(1-\alpha)$. The parameters of representations \eqref{Srep1} and \eqref{Srep2} can be connected: for $\alpha = 1$,  it holds that $\beta_2=\beta$ and $\sigma_2=2 \sigma/\pi$;  while for $\alpha \neq 1$ one has $\sigma$ and $\sigma_2$, $\beta$ and $\beta_2$  related by the equations
\begin{equation}
\tan\left(  \frac{\beta_2 \pi K(\alpha)}{2} \right) = \beta \tan\left(  \frac{ \pi \alpha}{2} \right), \quad \sigma_2 = \sigma \left(1+ \beta^2 \tan^2\left(  \frac{ \pi \alpha}{2} \right)\right)^{1/(2 \alpha)}.
\end{equation}

\citet{CMS76} introduced a fast algorithm for generating $\alpha$-stable rv; later \citet{Weron96} provided proof details about the algorithm; using when possible, for continuity, the notation established in \citet{Weron96}, define 
\begin{equation}\label{gammaW}
\gamma=\gamma(Z^{(1)}) = \pi \Phi(Z^{(1)}) - \pi/2\qquad \text{and} \quad W=W(Z^{(2)})= - \log\left(1-\Phi(Z^{(2)})\right).
\end{equation}
and let 
\begin{equation}\label{gamma0}
\gamma_0 = - \beta_2 \frac{\pi K(\alpha)}{ 2 \alpha}.
\end{equation}

Note that $\gamma \sim U\left(-\frac{\pi}{2},\frac{\pi}{2}\right)$,  a uniform r.v. in the interval $\left(-\frac{\pi}{2},\frac{\pi}{2}\right)$ and $W \sim E(1)$, an exponential rv with mean 1. 

For $\alpha \neq 1$ let $X=G_0(Z^{(1)},Z^{(2)})$ where
\begin{equation}\label{G}
G_0(z_1,z_2)= \frac{\sin(\alpha(\gamma(z_1)-\gamma_0))}{\left(\cos\gamma(z_1)\right)^{1/\alpha}}\left(\frac{\cos(\gamma(z_1)-\alpha (\gamma(z_1) -\gamma_0))}{W(z_2)}\right)^{(1-\alpha)/\alpha}; 
\end{equation}

for $\alpha=1$ let $X=G_1(Z^{(1)},Z^{(2)})$ where
\begin{equation}\label{G1}
G_1(z_1,z_2)= \left( \frac{\pi}{2}+\beta_2 \gamma(z_1) \right) \tan(\gamma(z_1)) - \beta_2 \log \left( \frac{W(z_2) \cos (\gamma(z_1))}{\pi/2+ \beta_2 \gamma(z_1)} \right).
\end{equation}

From \citet{CMS76}, \citet{Weron96} we have the following proposition:

\begin{prop}\label{prop:1} 
Let $\gamma$, $W$ and $\gamma_0$ be defined respectively as in \eqref{gammaW} and \eqref{gamma0}; let $G_0(\cdot)$ and $G_1(\cdot)$ be defined respectively as in \eqref{G} and \eqref{G1}. Then: for $\alpha \neq 1$, $X =G_0(Z^{(1)},Z^{(2)})$ is $S_\alpha(\beta_2,1,0)$ in the representation \eqref{Srep2}; for $\alpha=1$, $X=G_1(Z^{(1)},Z^{(2)})$ is $S_1(\beta_2,1,0)$ in the representation \eqref{Srep2}.
\end{prop}

Proposition \ref{prop:1} suffices for generating $S_\alpha(\beta,\sigma,\mu)$ rv as the class is invariant under affine transformations of the type $X \mapsto a X + b$, $a, b \in \Real$. More specifically, if $X \sim S_\alpha(\beta,1,0)$, then $Y \sim S_\alpha(\beta,\sigma,\mu)$ for
\begin{equation}
Y = \begin{cases}
\sigma X + \mu, \qquad \qquad \qquad \, \,   \alpha  \neq 1 \\
\sigma X + \frac{2}{\pi} \beta \sigma \log \sigma + \mu, \quad \alpha = 1 .
\end{cases}
\end{equation}
Finally we recall that, if $F(x,\alpha,\beta_2)$ represents the CDF of a $S_\alpha(\beta_2,1,0)$ r.v., for any admissible parameters $\alpha$ and $\beta_2$ (or $\beta$), the following equality holds
\begin{equation}\label{symm}
F(x,\alpha, \beta_2)=1 - F(-x, \alpha, -\beta_2), \quad x \in \Real.
\end{equation}

\subsection{Hermite polynomials expansion of the EP}

From the discussion in \ref{sub:Stablerv} it follows that we can represent the EDF of a stable rv as
\begin{equation}
F_n(x)= \frac{1}{n} \sum_{i=1}^n \ind{\{X_i\leq x \} }=\frac{1}{n} \sum_{i=1}^n \ind\{G_k(Z_i^{(1)},Z_i^{(2)})\leq x \}, \quad k=0,1,
\end{equation}
where $k=1$ if $\alpha =1$ and $k=0$ in all other cases $0 < \alpha <2$. We are not explicitly interested in the gaussian case as it can be solved directly in a much simpler way; indeed the transformation \eqref{G} reduces to the well known Box-Muller transformation for $\alpha=2$ and $\beta_2=0$.

Since the function $\ind\{(G_k(Z_i^{(1)},Z_i^{(2)})\leq x )\}$ $k=0,1$ is square integrable with respect to the standard gaussian density, we are going to provide an expansion of \eqref{EP} in terms of orthogonal Hermite polynomials.

Let $\phi(u)$, $u \in \Real$ denote the standard gaussian density and ${\cal L}_2={\cal L}_2(\Real^2, \phi(u) \phi(v) \, du \, dv)$ be the Hilbert space of real measurable functions $H(u,v)$ such that
\begin{equation}
\E \left[ H^2(u,v)\right] = \int_{\Real^2} H^2(u,v) \phi(u) \phi(v) \, du \, dv < \infty
\end{equation}
and let $H_m$ denote the standard Hermite polynomials, i.e.
$H_m(u) = (-1)^m \phi^{-1}(u) \frac{d^m}{d u^m} \phi(u)$.
Since the system $\{H_{m_1}(u)H_{m_2}(v)\}_{m_1 \geq 0, m_2 \geq 0}$ is a complete orthogonal system for ${\cal L}_2$, for every $x$ there exists an  expansion 
\begin{equation}\label{HPE}
\ind\{X_i\leq x \} =\ind \{G_k(Z_i^{(1)},Z_i^{(2)})\leq x \} = \sum_{m \geq 0} \,\, \sum_{m_1+ m_2 \geq m}  \frac{J_{m_1,m_2}^k(x)}{m_1! m_2!} H_{m_1}(Z_i^{(1)})H_{m_2}(Z_i^{(2)}), \quad k=0,1,
\end{equation} 
converging in ${\cal L}_2$ with coefficients
\begin{equation}\label{Jcoeff}
J_{m_1,m_2}^k(x) = \E^{Z^{(1)},Z^{(2)}}\left[\ind \{G_k(Z^{(1)},Z^{(2)})\leq x \} H_{m_1}(Z^{(1)})H_{m_2}(Z^{(2)})\right], \quad k=0,1.
\end{equation}
When not explicitly necessary, we will suppress dependence of the $J$'s coefficients and other quantities on $k$ and refer generally to an $S_\alpha(\beta_2,1,0)$ r.v., $0< \alpha < 2$ obtained via the transformation $G_1$ if $\alpha =1$ and $G_0$ otherwise.

Note that by a change of variable technique, from Proposition \ref{prop:1}, $J_{0,0}^k(x) = F(x)$ where $F$ indicates the CDF of a $S_\alpha(\beta_2,1,0)$ r.v., $0< \alpha < 2$. It follows that we have the ${\cal L}_2$ expansion
\begin{equation}
F_n(x) - F(x) = \sum_{q \geq m} \sum_{m_1 + m_2 =q} \frac{J_{m_1,m_2}(x)}{m_1! m_2!} \frac{1}{n} \sum_{i=1}^n H_{m_1}(Z_i^{(1)})H_{m_2}(Z_i^{(2)}).
\end{equation}

Define here $m=m(x)$ as the Hermite rank of the function $\ind\{G_1(u,v)\leq x\}$ (similarly for $G_0$), that is $m=m(x)=\min \{m_1+m_2=m: J_{m_1,m_2}(x) \neq 0 \}$. By the well known property of Hermite polynomials, with $\delta_{m}^{n}$ indicating Kronecker's delta, $\E H_{m}(Z_0)  H_{n}(Z_k)=\delta_{m}^{n} m!r^m(k)$, from which,
\begin{equation}
\V(F_n(x)) = \sum_{q \geq m} \sum_{m_1 + m_2 =q} \frac{[J_{m_1,m_2}(x)]^2}{m_1! m_2!} \, \sigma^2_{n,q}
\end{equation}
with 
\begin{equation}
\sigma^2_{n,q} = \frac{1}{n^2} \sum_{i=1}^n\sum_{j=1}^n \E \left[H_{m_1}(Z_i^{(1)})H_{m_1}(Z_j^{(1)})\right] \E \left[H_{m_2}(Z_i^{(2)})H_{m_2}(Z_j^{(2)})\right] = \frac{1}{n^2} \sum_{i=1}^n\sum_{j=1}^n r^{q}(|i-j|).
\end{equation}
For $0< D  < 1/m$ we obtain, as $N\rightarrow \infty$, that $\sigma^2_{n,m} \sim c(m,D) L^m(n) n^{-mD}$ with the constant $c(m, D)= 2[(1-mD)(2-mD)]^{-1}$, i.e. , if the rank of the expansion \eqref{HPE} is $m$, and  $0< D  < 1/m$ then the EP exhibits LRD. 

$F_n(x)$ can then be expressed as a bivariate expansion in Hermite polynomials. A uniform reduction principle as well as weak convergence results for this case are discussed by \citet{LeoSak01} and \citet{LeoSakTau02}, based on the results of \citet{Taqqu75}, \citet{Taqqu79}, \citet{DobMaj79} and \citet{DehTaq89} using a construction of multiple Wiener It\^o integrals with dependent integrators as proposed in \citet{FoxTaq87}.  These previous result are summarized in the following proposition:

\begin{prop} \label{prop:2} Let Assumption \ref{A:1} hold and the functions $\ind\{G_0(u,v)\leq x\}$ and $\ind\{G_1(u,v)\leq x\}$ have Hermite rank $m \geq 1$ and $0<D<1/m$. Let $d_{n,m}^2=c(m,D) \, n^{-mD}L^m(n)$ and define, for $t\in[0,1]$, 
$$ 
D([nt],x)=d^{-1}_{n,m}[nt]\left(F_{[nt]}(x) - F(x)\right)
$$
Then,
\begin{itemize}
\item[a)] 
D(n,x) converges, as $n \rightarrow \infty$, in ${\cal L}_2$ to 
\begin{equation}
d^{-1}_{n,m} \left(\sum_{m_1 + m_2 =m} \frac{J_{m_1,m_2}(x)}{m_1! m_2!} \frac{1}{n} \sum_{i=1}^n H_{m_1}(Z_i^{(1)})H_{m_2}(Z_i^{(2)})\right)
\end{equation}
\item[b)]
$\{D([nt],x); \, -\infty \leq x \leq \infty; \, 0\leq t\leq1\}$ converges, as $n \rightarrow \infty$ to the process
\begin{equation}
\left\{\sum_{m_1 + m_2 =m} \frac{J_{m_1,m_2}(x)}{m_1! m_2!} \, Z_{m_1,m_2}(t);\, -\infty \leq x \leq \infty; \, 0\leq t\leq1 \right\}
\end{equation}
in the sense of weak convergence in the space  $D[-\infty, \infty] \times [0,1]$, equipped with sup-norm.
\end{itemize}
\end{prop}
The processes $Z_{m_1,m_2}(t)$, $m_1,m_2 \geq 0$, $m_1+m_2=m$ are given as multiple Wiener-It\^o integrals of the form
\begin{equation}
Z_{m_1,m_2}(t) = K(m,D) \int_{\Real^m}^{'} \frac{e^{it(\lambda_1 + \dots + \lambda_m)}}{i(\lambda_1 + \dots + \lambda_m)} \prod_{j=1}^m |\lambda_j|^{(D-1)/2} \, \prod_{j=1}^{m_1} W_1(d \lambda_j) \, \prod_{j=m_1+1}^{m_2} W_2(d \lambda_j)
\end{equation}
where $W_1$ and $W_2$ are independent copies of a complex valued gaussian white noise on $\Real$ and 
\begin{equation}
K(m,D)= \frac{\frac{1}{2}(1-mD)(2-mD)}{\sqrt{m! \Gamma(D) \sin[(1-D)\pi/2]}}.
\end{equation}
The symbol of integration $\int_{\Real^m}^{'}$ stands to indicate that the hyper diagonals $\{ \lambda_j=\lambda_k, j \neq k \}$ are excluded form the domain of integration. Note that $Z_{m_1,m_2}(1)$ is gaussian for $m_1+m_2=1$ and that the normalizing factor $K(m,D)$ ensures unit variance of $Z_{m_1,m_2}(1)$.

\begin{rem}
As discussed in the introduction, one could consider the simpler non-linear transformation, for $Z$ satisfying Assumption \ref{A:1}, $X = G(Z)$ for $G=F^{-1} \circ \Phi$, in which case an $L_2$ expansion in terms of Hermite polinomials would result in 
\begin{equation}
F_n(x) - F(x) = \sum_{q \geq m}  \frac{J_{q}(x)}{q!} \frac{1}{n} \sum_{i=1}^n H_{q}(Z_i)
\end{equation}
with $G$ having Hermite rank $m=1$ since $H_1(Z)=\E [ \ind\{Z\leq F\circ \Phi^{-1}(x)\} Z] = - \phi(F\circ \Phi^{-1}(x))$. Although this approach would be much simpler for asymptotic analysis, the bivariate case will be considered in detail here for the reasons discussed in the introduction.
\end{rem}

\section{Hermite rank of the stable-EP}\label{Sec:HR}

For  $(Z^{(1)},Z^{(2)})=(Z_1,Z_2)$ satisfying Assumption 1 (indeed only normality and independence are exploited) the main result of this section is the proof that the functions $\ind\{G_0(Z_1,Z_2)\leq x\}$ and $\ind\{G_1(Z_1,Z_2)\leq x\}$ have Hermite rank $m=1$ $\forall x$ and consequently the asymptotic distribution of \eqref{EP}, properly normalized, is gaussian. Explicit formulae for the coefficients are presented. As there are several cases, the result is presented in three separate theorems which discuss respectively the cases $0<\alpha<1$, $\alpha=1$, $1 < \alpha <2$. 

Since symmetry relations \eqref{symm} will be exploited in deriving the coefficients $J_{m_1,m_2}(x)$, their dependence on $\beta_2$  will be explicitly outlined by writing $J_{m_1,m_2}(x,\beta_2)$. 

Also, let
\begin{equation}\label{agamma}
a(\gamma)= \left( \frac{\sin \alpha(\gamma-\gamma_0)}{\cos \gamma} \right)^{\frac{\alpha}{(1-\alpha)}} \, \frac{\cos (\gamma-\alpha(\gamma-\gamma_0))}{\cos \gamma},
\end{equation}

\begin{equation}\label{a1gamma}
a_1(\gamma)=  \frac{\frac{\pi}{2} +\beta_2 \gamma}{\cos \gamma}  \, \exp \left\{ \frac{1}{\beta_2} \left(\frac{\pi}{2}+ \beta_2 \gamma \right) \tan \gamma \right\}.
\end{equation}

\begin{theorem}\label{th:1}
Let $0<\alpha<1$; the function $\ind\{G_0(Z_1,Z_2)\leq x\}$  has Hermite rank $m=m(x)=1$ $\forall x \in (-\infty,\infty)$ with coefficients:
\begin{itemize}
\item[a)] for $x >0$,
\begin{equation}\label{T1a1}
J_{1,0}(x, \beta_2) = \frac{1}{\pi} \int_{\gamma_0}^{\pi/2} e^{-x^{\frac{\alpha}{\alpha-1}} a(\gamma)} \Phi^{-1}\left(\frac{1}{\pi}(\gamma+ \frac{\pi}{2})\right) \, d \, \gamma - (\phi \circ \Phi^{-1})\left(\frac{1}{\pi}(\gamma_0+ \frac{\pi}{2})\right),
\end{equation}
\begin{equation}\label{T1a2}
J_{0,1}(x, \beta_2) = \frac{1}{\pi} \int_{\gamma_0}^{\pi/2} ( \phi \circ \Phi^{-1})(1-e^{-x^{\frac{\alpha}{\alpha-1}} a(\gamma)} )  \, d \, \gamma;
\end{equation}
\item[b)] for $x =0$,
\begin{equation}
J_{1,0}(0, \beta_2) =  - (\phi \circ \Phi^{-1})\left(\frac{1}{\pi}(\gamma_0+ \frac{\pi}{2})\right),
\end{equation}
\begin{equation}
J_{0,1}(0, \beta_2) = 0;
\end{equation}
\item[c)] for $x <0$, formulae can be derived from the case $x>0$: for $J_{1,0}(x, \beta_2)$, using formula \eqref{T1a1}, compute $J_{1,0}(-x, -\beta_2)$ while for  $J_{0,1}(x, \beta_2)$, using formula \eqref{T1a2}, compute $-J_{0,1}(-x, -\beta_2)$.
\end{itemize}
\end{theorem}

\begin{theorem}\label{th:2}
Let $\alpha=1$; the function $\ind\{G_1(Z_1,Z_2)\leq x\}$  has Hermite rank $m=m(x)=1$ $\forall x \in (-\infty,\infty)$ with coefficients:
\begin{itemize}
\item[a)] if $\beta_2=0$,
\begin{equation}\label{T2a1}
J_{1,0}(x, 0) =  - (\phi \circ \Phi^{-1})\left[\frac{1}{\pi}\left(\arctan \left( \frac{2}{\pi} x \right)+\frac{\pi}{2}\right)\right],
\end{equation}
\begin{equation}
J_{0,1}(x, 0) = 0;
\end{equation}
\item[b)] if $\beta_2>0$,
\begin{equation}\label{T2b1}
J_{1,0}(x, \beta_2) = \frac{1}{\pi} \int_{-\pi/2}^{\pi/2} \exp\{-e^{-x/\beta_2} a_1(\gamma)\} \, \Phi^{-1}\left(\frac{1}{\pi}(\gamma+ \frac{\pi}{2})\right) \, d \, \gamma ,
\end{equation}
\begin{equation}\label{T2b2}
J_{0,1}(x, \beta_2) = \frac{1}{\pi} \int_{-\pi/2}^{\pi/2} ( \phi \circ \Phi^{-1})(1-\exp\{-e^{-x/\beta_2} a_1(\gamma)\} )  \, d \, \gamma;
\end{equation}
\item[c)] if $\beta_2 <0$, formulae can be derived from the case $\beta_2>0$: for $J_{1,0}(x, \beta_2)$, using formula \eqref{T2b1}, compute $J_{1,0}(-x, -\beta_2)$ while for  $J_{0,1}(x, \beta_2)$, using formula \eqref{T2b2}, compute $-J_{0,1}(-x, -\beta_2)$. 
\end{itemize}
\end{theorem}

\begin{theorem}\label{th:3}
Let $1<\alpha<2$; the function $\ind\{G_0(Z_1,Z_2)\leq x\}$  has Hermite rank $m=m(x)=1$ $\forall x \in (-\infty,\infty)$ with coefficients:
\begin{itemize}
\item[a)] for $x \geq 0$,
\begin{equation}\label{T3a1}
J_{1,0}(x, \beta_2) = - \frac{1}{\pi} \int_{\gamma_0}^{\pi/2} e^{-x^{\frac{\alpha}{\alpha-1}} a(\gamma)} \Phi^{-1}\left(\frac{1}{\pi}(\gamma+ \frac{\pi}{2})\right) \, d \, \gamma ,
\end{equation}
\begin{equation}\label{T3a2}
J_{0,1}(x, \beta_2) = \frac{1}{\pi} \int_{\gamma_0}^{\pi/2} ( \phi \circ \Phi^{-1})(1-e^{-x^{\frac{\alpha}{\alpha-1}} a(\gamma)} )  \, d \, \gamma;
\end{equation}
\item[b)] for $x <0$, formulae can be derived from the case $x>0$: for $J_{1,0}(x, \beta_2)$, using formula \eqref{T3a1}, compute $J_{1,0}(-x, -\beta_2)$ while for  $J_{0,1}(x, \beta_2)$, using formula \eqref{T3a2}, compute $-J_{0,1}(-x, -\beta_2)$.
\end{itemize}
\end{theorem}

\begin{rem}
The formulae presented in the theorems can be seen as a generalization of integral representations discussed in \citet{Zol86} and \citet{Weron96}. From the numerical point of view they are quite fast to calculate although some parameter values could easily induce overflow; in the supplemental material this issue will be discussed in more detail. 
\end{rem}

Before proving the theorems, recall the  definition of $J_{1,0}(x, \beta_2)$ and  $J_{0,1}(x, \beta_2)$ from \eqref{Jcoeff}. Also, if needed,  dependence of $G_0(z_1,z_2)$ on $\beta_2$ will be highlighted by writing $G_0(z_1,z_2,\beta_2)$.

\begin{proof}[Proof of Theorem \ref{th:1}]

Note that one can write 
\begin{equation}
\ind \{G_0(z_1,z_2) \leq x \} = \ind \{G_0(z_1,z_2) \leq x \}\left[ \ind \{\gamma(z_1)>\gamma_0 \}+\ind\{\gamma(z_1)\leq\gamma_0\}\right]
\end{equation}
and that  $G_0(z_1,z_2) >0$ if and only if $\gamma(z_1) > \gamma_0$. Consider first the case $x>0$; from the reasoning above it follows that, 
\begin{itemize}
\item[i)]  $\ind\{G_0(z_1,z_2) \leq x \}\ind \{\gamma(z_1)>\gamma_0 \}=\ind\{0 < G_0(z_1,z_2) \leq x \}\ind \{\gamma(z_1)>\gamma_0 \}$,  $x >0$;
\item[ii)] $\ind\{G_0(z_1,z_2) \leq x \}\ind \{\gamma(z_1)\leq \gamma_0 \}=\ind\{G_0(z_1,z_2) \leq 0 \}\ind \{\gamma(z_1)\leq \gamma_0 \}=\ind \{\gamma(z_1)\leq \gamma_0 \}$,  $x>0$.
\end{itemize}
To determine $J_{1,0}(x, \beta_2)$ in case $a)$, $x>0$, we then need to compute
\begin{equation}
\begin{split}
J_{1,0}(x, \beta_2)&= \int_{\Real^2} \ind\{0 < G_0(z_1,z_2) \leq x \}\ind \{\gamma(z_1)>\gamma_0 \} \, z_1 \phi(z_1)\phi(z_2) \, d z_1 \, dz_2 \\
                   & \qquad \qquad  \qquad \qquad + \int_{\Real^2} \ind\{\gamma(z_1)\leq\gamma_0\}  \, z_1 \phi(z_1)\phi(z_2) \, d z_1 \, dz_2 
\end{split}
\end{equation}
Denote the two integrals on the $r.h.s.$ of the above equation as $I_1 + I_2$. As far as $I_1$ is concerned, since $(1-\alpha)/\alpha >0$ we can write (see formula details in \eqref{G}),  $\ind\{0 < G_0(z_1,z_2) \leq x \}= \ind \{W(z_2) \geq x^{\alpha/(\alpha-1)} a(\gamma(z_1)) \}$; then after making the transformation $W=W(z_2)=- \log (1- \Phi(z_2))$ we have
\begin{equation}
\begin{split}
I_1 &= \int_{\Real^2} \ind \{w \geq x^{\alpha/(\alpha-1)} a(\gamma(z_1)) \}\ind \{\gamma(z_1)>\gamma_0 \} \, z_1 \phi(z_1) \, e^{-w} \, dw \, dz_1 \\
    &= \int_\Real e^{-x^{\alpha/(\alpha-1)} a(\gamma(z_1))}\ind \{\gamma(z_1)>\gamma_0 \} \, z_1 \phi(z_1)\, dz_1 \\
    &= \frac{1}{\pi} \int_{\gamma_0}^{\pi/2} e^{-x^{\alpha/(\alpha-1)} a(\gamma)} \phi^{-1}\left(\frac{1}{\pi}(\gamma+\pi/2)\right) \, d \gamma
\end{split}
\end{equation}
where the last step has been obtained by the transformation $\gamma=\gamma(z_1)=\pi \Phi(z_1)-\pi/2$.

As far as $I_2$ is concerned, it reduces to computing 
\begin{equation}
\int_\Real \ind \{\gamma(z_1)\leq \gamma_0 \} \, z_1 \phi(z_1)\, dz_1 = \int_\Real \ind \{z_1 \leq \gamma^{-1}(\gamma_0) \} \, z_1 \phi(z_1)\, dz_1 = \phi(\gamma^{-1}(\gamma_0))
\end{equation}
where $\gamma^{-1}(\gamma_0)= \Phi^{-1}\left(\frac{1}{\pi}(\gamma_0+\pi/2) \right)$. Putting together the results for $I_1$ and $I_2$ yields the coefficient $J_{1,0}(x, \beta_2)$ in case $a)$, $x>0$. 

$J_{1,0}(x, \beta_2)$ in case $b)$, $x=0$,  is simply obtained by $I_2$. 

To determine $J_{1,0}(x, \beta_2)$ in case $c)$, $x<0$, note that since, for expectation taken with respect to $Z_1$ and $Z_2$, 
$\E[Z_1]=0$, then, for any $x \in \Real$, 
\begin{equation}\label{EVa}
\begin{split}
\E \left[ Z_1 \ind\{G_0(z_1,z_2,\beta_2) \leq x \}\right] &= -\E \left[ Z_1 \ind\{G_0(z_1,z_2,\beta_2) > x \}\right] \\
                                                          &=-\E \left[ Z_1 \ind\{-G_0(z_1,z_2,\beta_2) < - x \}\right]. 
\end{split}
\end{equation}    
Given the definition of $G_0$ in \eqref{G}, we note that, for $x \in [-\pi/2, \pi/2]$, $-\sin(x)=\sin(-x)$, $\cos(x)=\cos(-x)$ and that $-\gamma(-Z_1)\sim U[-\pi/2,\pi/2]$. It follows that $\E[Z_1 \ind\{-G_0(z_1,z_2,\beta_2) <- x \}] = - \E[Z_1 \ind\{G_0(-z_1,z_2,-\beta_2) < -x]$ and hence, substituting into \eqref{EVa}, we have
$$\E[Z_1 \ind\{G_0(z_1,z_2,\beta_2) \leq x \}] =  \E[Z_1 \ind\{G_0(z_1,z_2,-\beta_2) \leq -x]$$ 
from which the statement for case $c)$, $x<0$, of the theorem.

\smallskip
Consider now computation of $J_{0,1}(x, \beta_2)$ in the case $x>0$, similarly to what done for $J_{1,0}(x, \beta_2)$ we need to compute
\begin{equation}
\begin{split}
J_{0,1}(x, \beta_2)&= \int_{\Real^2} \ind\{0 < G_0(z_1,z_2) \leq x \}\ind \{\gamma(z_1)>\gamma_0 \} \, z_2 \phi(z_1)\phi(z_2) \, d z_1 \, dz_2 \\
                   & \qquad \qquad  \qquad \qquad + \int_{\Real^2} \ind\{\gamma(z_1)\leq\gamma_0\} \, z_2 \phi(z_1)\phi(z_2) \, d z_1 \, dz_2 
\end{split}
\end{equation}
where we note this time that the second integral on the $r.h.s.$ of the above formula is null. We then compute simply
\begin{equation}
\begin{split}
J_{0,1}(x, \beta_2)&= \int_{\Real^2} \ind \{w \geq x^{\alpha/(\alpha-1)} a(\gamma(z_1)) \}\ind \{\gamma(z_1)>\gamma_0 \}  \Phi^{-1}(1-e^{-w}) e^{-w}  \phi(z_1) \, d w \, d z_1  \\
                   &= \int_{\Real} \ind\{\gamma(z_1)> \gamma_0\} ( \phi \circ \Phi^{-1})(1-e^{-x^{\frac{\alpha}{\alpha-1}} a(\gamma(z_1))} ) \phi(z_1) \, d z_1  \\
                   &= \frac{1}{\pi} \int_{\gamma_0}^{\pi/2}( \phi \circ \Phi^{-1})(1-e^{-x^{\frac{\alpha}{\alpha-1}} a(\gamma)} ) \, d  \gamma
\end{split}
\end{equation}
where, as before, the transformations $W=W(z_2)=- \log (1- \Phi(z_2))$ and $\gamma =\gamma(z_1)=\pi \Phi(z_1)-\pi/2 \,$ have been used in turn. From the results above and same reasoning as  for the case $J_{1,0}(0, \beta_2)$, we have $J_{0,1}(0, \beta_2)=0$. In the case $x<0$, a parallel reasoning to the corresponding case $J_{1,0}(x, \beta_2)$, yields that $\E[Z_2 \ind\{G_0(z_1,z_2,\beta_2) \leq x \}] = - \E[Z_2 \ind\{G_0(z_1,z_2,-\beta_2) \leq -x]$ from which the statement case $c)$, $x<0$ of the theorem.
\end{proof}

\begin{proof}[Proof of Theorem \ref{th:2}] Consider  the case $\beta_2=0$, in which $G_1(z_1,z_2)$ reduces to $\frac{\pi}{2} \tan [\gamma(z_1)]$  and, for $\gamma(z)$ defined in \eqref{gammaW}, 
$
J_{1,0}(x, \beta_2) = \E \left[ Z_1 \ind\{\frac{\pi}{2} \tan [\gamma(Z_1)] \leq x \} \right]
$
reduces to \eqref{T2a1} and $J_{0,1}(x, \beta_2) = \E \left[ Z_2 \ind\{\frac{\pi}{2} \tan [\gamma(Z_1)] \leq x \} \right]=0$.

In the case where $\beta_2 \neq 0$, $G_1(z_1,z_2)$ reduces to $\beta_2 \log[a_1(\gamma(z_1)/W(z_2)]$ with $a_1$ defined in \eqref{a1gamma}. Hence, for $\beta_2 >0$, $x \in \Real$, using the  transformations $W=W(z_2)=- \log (1- \Phi(z_2))$ and $\gamma =\gamma(z_1)=\pi \Phi(z_1)-\pi/2$,
\[
\begin{split}
J_{1,0}(x, \beta_2)&= \E \left[ Z_1 \ind\{W(Z_2) \geq e^{-x/\beta_2} a_1(\gamma(Z_1)) \} \right] \\
                   &= \int_\Real \exp\left\{ -e^{-x/\beta_2} a_1(\gamma(z_1))\right\} \, z_1  \phi(z_1) \, d  z_1 \\
                   \end{split}
\]
which reduces to \eqref{T2b1}, and
\[
\begin{split}
J_{0,1}(x, \beta_2)&= \E \left[ Z_2 \ind\{W(Z_2) \geq e^{-x/\beta_2} a_1(\gamma(Z_1)) \} \right] \\
                   &= \int_{\Real^2} \ind\{w \geq e^{-x/\beta_2} a_1(\gamma(z_1)) \} \Phi^{-1}(1-e^{-w})  e^{-w}  \, \phi(z_1) \, d  w \, d  z_1 \\
                   &= \int_\Real ( \phi \circ \Phi^{-1})(1-\exp\{e^{-x/\beta_2} a_1(\gamma(z_1))\} )\, \phi(z_1)\, d  z_1. 
 \end{split}
\]
As far as the case $\beta_2<0$, parallel reasoning exploiting symmetries, as done in the proof of Theorem \ref{th:1} brings to result $c)$ in Theorem \ref{th:2}.
\end{proof}

\begin{proof}[Proof of Theorem \ref{th:3}]
Following a similar scheme of proof as in Theorem \ref{th:1}, consider first the case $x>0$ and note that, since $\E[Z_1]=0$, where expectation is taken $wrt$ $Z_1$ and $Z_2$, 
\begin{equation}
\begin{split}
\E[Z_1 \ind \{G_0(Z_1,Z_2) \leq x \}] &= -\E [Z_1 \ind \{G_0(Z_1,Z_2)> x \}] \\
                                      &= -E\left[Z_1 \ind \{G_0(Z_1,Z_2) > x \}\left[ \ind \{\gamma(Z_1)>\gamma_0 \}+\ind\{\gamma(Z_1)\leq\gamma_0\}\right] \right] \\
                                      &= -E\left[Z_1 \ind \{G_0(Z_1,Z_2) > x \}\left[ \ind \{\gamma(Z_1)>\gamma_0 \}\right] \right], \quad x>0, 
\end{split}
\end{equation}
since $G_0$ cannot the greater than $x>0$ when $\gamma(z_1) \leq \gamma_0$. Since, for $1<\alpha<2$, $(\alpha-1)/\alpha >0$, for $x>0$ we can make the following computations:
\begin{equation}
\begin{split}
J_{1,0}(x,\beta_2) &=- \int_{\Real^2} \ind \{G_0(z_1,z_2) > x \} \ind \{\gamma(z_1)>\gamma_0 \} \, z_1 \, \phi(z_1) \phi(z_2) \, d  z_1 \, d  z_2 \\
                   &=- \int_{\Real^2} \ind \{W(z_2) > x^{\alpha/(\alpha-1)} a(\gamma(z_1)) \} \ind \{\gamma(z_1)>\gamma_0 \} \, z_1 \, \phi(z_1) \phi(z_2) \, d  z_1 \, d  z_2 \\
                   &=-\int_{\Real} e^{-x^{\alpha/(\alpha-1)} a(\gamma(z_1))} \ind \{\gamma(z_1)>\gamma_0 \} \, z_1 \, \phi(z_1)  \, d  z_1
\end{split}
\end{equation}
which reduces to \eqref{T3a1} after transforming $\gamma =\gamma(z_1)=\pi \Phi(z_1)-\pi/2$. The case for $x<0$ can be recovered by symmetry, following a parallel reasoning as the one in the proof of Theorem \ref{th:1}.

As far as the second coefficient, $J_{0,1}(x, \beta_2)$ is concerned, again, following the discussion above, for $x>0$ we can make the following computations:
\begin{equation}
\begin{split}
J_{0,1}(x,\beta_2) &=- \int_{\Real^2} \ind \{G_0(z_1,z_2) > x \} \ind \{\gamma(z_1)>\gamma_0 \} \, z_2 \, \phi(z_1) \phi(z_2) \, d  z_1 \, d  z_2 \\
                   &=- \int_{\Real^2} \ind \{W(z_2) > x^{\alpha/(\alpha-1)} a(\gamma(z_1)) \} \ind \{\gamma(z_1)>\gamma_0 \} \, z_2 \, \phi(z_1) \phi(z_2) \, d  z_1 \, d  z_2 \\
                   &=- \int_{\Real^2} \ind \{w > x^{\alpha/(\alpha-1)} a(\gamma(z_1)) \} \ind \{\gamma(z_1)>\gamma_0 \} \, \Phi^{-1}(1-e^{-w}) e^{-w} \, \phi(z_1)  \, d  w \, d  z_1 \\
                   &=- \int_{\Real} \ind \{\gamma(z_1)>\gamma_0 \} \,(\phi \circ \Phi^{-1})(1-\exp\{- x^{\alpha/(\alpha-1)} a(\gamma(z_1)) \}  \, \phi(z_1)  \, d  z_1 
\end{split}
\end{equation}
which reduces to \eqref{T3a2} after transforming $\gamma =\gamma(z_1)=\pi \Phi(z_1)-\pi/2$. The case for $x<0$ can be recovered by symmetry, following a parallel reasoning as the one in the proof of Theorem \ref{th:1}.
\end{proof}


\section{Application to goodness-of-fit testing}

As an application of the results of the last section, we consider the problem of testing the simple hypothesis $H_0: F=F_0$ for $F_0$ in the class of $\alpha$-stable distributions with $0<\alpha<2$ when the data show LRD as defined in the previous sections. The Kolmogorov-Smirnov statistic 
\begin{equation}
K_n= \sup_{x \in \Real }| F_n(x)-F(x)|
\end{equation}
will be discussed in some detail. For a stable rv $X$ defined as in Proposition \ref{prop:1} with $(Z^{(1) },Z^{(2) })$ satisfying Assumption \ref{A:1}, Proposition \ref{prop:2} implies that 
\begin{equation}
\sup_{x \in \Real }d_{n,1}^{-1}| (F_n(x)-F(x))-(J_{1,0}(x)Z_{1,0}(1)+ J_{0,1}(x)Z_{0,1}(1))|=o_P(1).
\end{equation}
Since $Z_{1,0}(1)$ and $Z_{0,1}(1)$ are two independent standard normal rv, one readily obtains that, under $H_0$, 
\begin{equation}\label{KST}
d_{n,1}^{-1} \frac{K_n}{c_0}\rightarrow_D |Z|, \qquad c_0 = \sup_{x \in \Real }\sqrt{(J_{1,0}(x))^2+ (J_{0,1}(x))^2}
\end{equation}
where $\rightarrow_D$ means convergence in distribution and $Z$ is a standard normal random variable.  It is worth emphasizing that such a simple and appealing result for the KS statistics based on $\alpha$-stable rv with LRD has never been derived in the literature. For analogous results for long memory moving averages see \citet{KouSur10} and the reference therein which however do not include the stable case.

Similar results will be obtained for any other test based on continuous functionals of the first order difference $d_{n,1}^{-1} (F_n(x)-F(x))$ such as the Cram\'er-von Mises test which will obtain an asymptotic distribution related to a $\chi^2$-distribution with one degree of freedom. 

These results are in sharp contrast with those of the i.i.d. setting. An noted by \citet{KouSur10} however, the test \eqref{KST} cannot distinguish $n^{1/2}$-neighborhoods of $F_0$; see \citet{KouSur10}, p. 3745, for furhter details which will not be repeated here.

In order to appreciate the precision of the asymptotic approximation a small Monte Carlo study  where the data generated satisfy the set up defined in Section \ref{Sec:Background} is performed. In order to implement the Monte Carlo experiment the following steps are taken (for further details see the supplemental material): 
\begin{itemize}
\item[i)] generate two random sequences $(Z_1,Z_2)$ satisfying Assumption \ref{A:1} with covariance function $r(k)= (1+k^2)^{-D/2}$. Note that we can write $r(k)= k^{-D}L(k)$ with $L(k) = k^D (1+k^2)^{-D/2}$.

\item[ii)] Apply transformations \eqref{G} (or \eqref{G1}) to the above sequences;

\item[iii)] Compute the empirical process and the KS  statistics.
\end{itemize}

\begin{table}
{\small
	\centering
	\begin{tabular}{cccclccclccc}
\toprule 
&&&  &       & \multicolumn{3}{c}{$d^{-1}_nK_n/c_0$} &  & \multicolumn{3}{c}{$K_n^{sd}$}  \\
\cmidrule{6-8}
\cmidrule{10-12}
\textbf{D}&\textbf{n}&\textbf{m} & \textbf{sd} & $\boldsymbol{\gamma \rightarrow}$      & \textbf{0.8} &  \textbf{0.9} &  \textbf{0.95} &    & \textbf{0.8} &  \textbf{0.9}  & \textbf{0.95} \\
\midrule
 0.2 & 128 &  1.1837&0.5484&    & 0.6234&0.7894&0.8960     && 0.7810&0.8944&0.9514     \\
     & 256& 1.1165&0.5583&     &0.6676&0.8198&0.9062     && 0.7928&0.8922&0.9494 \\
		&512&	1.0714&0.5513&	&0.6872&0.8380&0.9208		&& 0.7908&0.8974&0.9496 \\
     &1024& 1.0265&0.5520& &0.7214&0.8586&0.9276   && 0.8022&0.8964&0.9480  \\
     &2048& 1.0101&0.5515& &0.7290&0.8616&0.9266   && 0.804&0.8944&0.9458\\
\midrule
 0.5 & 128 &1.1019&0.5365& & 0.6856&0.8392&0.9194   &&0.7998&0.8944&0.9474\\
     & 256 &1.0525&0.5499& & 0.7146&0.8546&0.9234   &&0.8096&0.897&0.9482\\
    &512  &1.0385&0.5503&  & 0.7126&0.8562&0.9308   &&0.7972&0.9000&0.9520\\
     &1024 &0.9995&0.5532& & 0.7362&0.8710&0.9344   &&0.8084&0.9028&0.9498 \\
     &2048&0.9823&0.5498&  & 0.7374&0.8694&0.9404   &&0.7938&0.8950&0.9516 \\
\midrule
 0.8 &128& 0.9505&0.4444&  &0.7898&0.9208&0.9688    &&0.8030&0.9032&0.9510\\
     &256& 0.9540&0.4546&  &0.7936&0.9140&0.9650    &&0.8114&0.9004&0.9466 \\
     &512&0.9541&0.4798&   &0.7824&0.9036&0.9576    &&0.8042&0.8996&0.9486\\
     &1024&0.9537&0.4921&  &0.7712&0.9006&0.9564    &&0.8026&0.9008&0.9474 \\
     &2048&0.9604&0.4915&  &0.7684&0.8994&0.9564    &&0.8008&0.9016&0.9506 \\
\bottomrule
\end{tabular}
		\caption{\footnotesize Monte Carlo estimates (N=5000) of mean, standard deviation and the theoretical probability $\gamma=P(|Z| \leq z_{\gamma/2})$ for $d^{-1}_nK_n/c_0$ and $K_n^{sd}$ (see respectively \eqref{KST} and \eqref{ESKS}) for selected values of $n$ and $D$ based on the EP constructed from a Stable rv with $\alpha=0.5$; $\beta_2=0.5$. }
		\label{tab:1}
	}
\end{table}

\begin{table}
{\small
	\centering
	\begin{tabular}{cccclccclccc}
\toprule 
&&&  &       & \multicolumn{3}{c}{$d^{-1}_nK_n/c_0$} &  & \multicolumn{3}{c}{$K_n^{sd}$}  \\
\cmidrule{6-8}
\cmidrule{10-12}
\textbf{D}&\textbf{n}&\textbf{m} & \textbf{sd} & $\boldsymbol{\gamma \rightarrow}$      & \textbf{0.8} &  \textbf{0.9} &  \textbf{0.95} &    & \textbf{0.8} &  \textbf{0.9}  & \textbf{0.95} \\
\midrule
 0.2 & 128 & 1.1242&0.5613&   & 0.6508&0.8198&0.9074  && 0.7938&0.8964&0.9510     \\
     & 256 & 1.0714&0.5681&   & 0.6826&0.8308&0.9142  && 0.7890&0.8902&0.9478 \\  
     & 512 & 1.0267&0.5655&   & 0.7124&0.8478&0.9258  && 0.7912&0.8944&0.9486 \\
     & 1024& 0.9873&0.5679&   & 0.7254&0.8600&0.9310  && 0.7942&0.8934&0.9494 \\
     & 2048& 0.9494&0.5755&   & 0.7498&0.8718&0.9308  && 0.8014&0.8956&0.9454 \\
\midrule
 0.5 & 128 & 1.0596&0.5509&   & 0.7032&0.8516&0.9250  && 0.8040&0.9014&0.9452 \\
     & 256 & 1.0122&0.5490&   & 0.7202&0.8636&0.9344  && 0.7976&0.8988&0.9510 \\ 
     & 512 & 0.9911&0.5709&   & 0.7296&0.8628&0.9320  && 0.7996&0.9000&0.9494 \\
     & 1024& 0.9446&0.5633&   & 0.7546&0.8762&0.9372  && 0.8036&0.8964&0.9466 \\
     & 2048& 0.9164&0.5707&   & 0.7618&0.8814&0.9466  && 0.7944&0.9002&0.9534 \\
\midrule
 0.8 & 128 & 0.9196&0.4440&   & 0.8064&0.9278&0.9708  && 0.8028&0.9036&0.9488 \\
     & 256 & 0.9128&0.4656&   & 0.8066&0.9180&0.9652  && 0.8104&0.8998&0.9478 \\
     & 512 & 0.9154&0.4798&   & 0.7964&0.9114&0.9638  && 0.8032&0.8978&0.9506 \\
     &1024 & 0.9148&0.4936&   & 0.7912&0.9070&0.9606  && 0.8028&0.8978&0.9470 \\
     &2048 & 0.9192&0.5060&   & 0.7878&0.9064&0.9548  && 0.8056&0.9028&0.9464 \\
\bottomrule
\end{tabular}
		\caption{\footnotesize MonteCarlo estimates (N=5000) of mean, standard deviation and the theoretical probability $\gamma=P(|Z| \leq z_{\gamma/2})$ for $d^{-1}_nK_n/c_0$ and $K_n^{sd}$ (see respectively \eqref{KST} and \eqref{ESKS}) for selected values of $n$ and $D$ based on the EP constructed from a Stable rv with $\alpha=1$; $\beta_2=0$. }
		\label{tab:2}
	}
\end{table}

\begin{table}
{\small
	\centering
	\begin{tabular}{cccclccclccc}
\toprule 
&&&  &       & \multicolumn{3}{c}{$d^{-1}_nK_n/c_0$} &  & \multicolumn{3}{c}{$K_n^{sd}$}  \\
\cmidrule{6-8}
\cmidrule{10-12}
\textbf{D}&\textbf{n}&\textbf{m} & \textbf{sd} & $\boldsymbol{\gamma \rightarrow}$      & \textbf{0.8} &  \textbf{0.9} &  \textbf{0.95} &    & \textbf{0.8} &  \textbf{0.9}  & \textbf{0.95} \\
\midrule
 0.2 & 128 & 1.1646&0.5595&   & 0.6178&0.8006&0.9016  && 0.7866&0.8998&0.9562     \\
     & 256 & 1.1182&0.5767&   & 0.6478&0.8180&0.9094  && 0.7924&0.9036&0.9534 \\
     &512 &  1.0598&0.5579&   & 0.6868&0.8426&0.9268  && 0.7954&0.9036&0.9528 \\
     &1024 & 1.0232&0.5654&   & 0.7020&0.8524&0.9292  && 0.7856&0.8972&0.9562 \\
     &2048 & 1.0101&0.5706&   & 0.7090&0.8512&0.9300  && 0.7900&0.8954&0.9526 \\
\midrule
 0.5 & 128 & 1.0895&0.5392&   & 0.6938&0.8472&0.9258  && 0.8010&0.9004&0.9526 \\
     & 256 & 1.0631&0.5464&   & 0.7062&0.8468&0.9248  && 0.7978&0.9018&0.9494 \\
     & 512 & 1.0306&0.5588&   & 0.7148&0.8566&0.9270  && 0.8032&0.9022&0.9514 \\
     &1024 & 1.0122&0.5534&   & 0.7264&0.8656&0.9322  && 0.8050&0.9020&0.9502 \\
     &2048 & 0.9988&0.5556&   & 0.7294&0.8698&0.9368  && 0.8008&0.9036&0.9498 \\
\midrule    
 0.8 & 128 & 0.9448&0.4433&   & 0.7962&0.9210&0.9686  && 0.8052&0.9010&0.9486 \\
     & 256 & 0.9519&0.4598&   & 0.7872&0.9160&0.9652  && 0.8064&0.9060&0.9498 \\
     & 512 & 0.9635&0.4782&   & 0.7628&0.9022&0.9618  && 0.7954&0.9012&0.9524 \\
     &1024 & 0.9695&0.4891&   & 0.7674&0.8976&0.9558  && 0.8010&0.9008&0.9486 \\
     &2048 & 0.9738&0.5079&   & 0.7632&0.8856&0.9484  && 0.8058&0.8962&0.9472  \\ 
\bottomrule
\end{tabular}
		\caption{\footnotesize MonteCarlo estimates (N=5000) of mean, standard deviation and the theoretical probability $\gamma=P(|Z| \leq z_{\gamma/2})$ for $d^{-1}_nK_n/c_0$ and $K_n^{sd}$ (see respectively \eqref{KST} and \eqref{ESKS}) for selected values of $n$ and $D$ based on the EP constructed from a Stable rv with $\alpha=1.5$; $\beta_2=0.8$. }
		\label{tab:3}
	}
\end{table}

Tables \ref{tab:1} to \ref{tab:3} contain the summary of three experiments analyzing the asymptotic distribution of the KS statistic respectively for the case where $X \sim S_{0.5}(0.5,1,0)$, $X \sim S_{1}(0,1,0)$ and $X \sim S_{1.5}(0.8,1,0)$. Each case, defined by sample size ($n=128$, $256$, $512$, $1024$, $2048$) was replicated $N=5000$ times. If we define with $K_{n,i}$ $i=1, \dots, N$ the $i$-th KS statistic obtained by an EP constructed on $n$ generated stable rv and $K_{n,i}^*= d_{n,1}^{-1} \frac{K_{n,i}}{c_0}$, i.e. the theoretically-standardized version of the KS statistic, in the tables below the following quantities are reported: 

\begin{itemize}
\item[a)] the mean and the standard deviation, simply computed as 
\begin{equation}
m = \frac{1}{N} \sum_{i=1}^N K_{n,i}^*, \qquad sd= \sqrt{\frac{1}{N} \sum_{i=1}^{N-1} (K_{n,i}^*-m)^2};
\end{equation}
\item[b)] the empirical probability $P(d_{n,1}^{-1} \frac{K_n}{c_0} \leq z_{\gamma/2})$  where $z_\gamma$ is the $\gamma$ percentile of the standard normal distribution, i.e., if $Z \sim N(0,1)$, then $P(Z \leq z_\gamma)=\gamma$; 
\item[c)] the empirical probability  $P(K_n^{sd} \leq z_{\gamma/2})$ where $K_n^{sd}$ is the empirically standardized version of the KS statistic adjusted to the theoretical mean and variance of the rv $|Z|$, i.e.
\begin{equation}\label{ESKS}
K_{n,i}^{sd} = \frac{(K_{n,i}^*-m)}{sd} \frac{\sqrt{\pi-2}}{\sqrt{\pi}} + \sqrt{\frac{2}{\pi}}
\end{equation}
where $\E|Z|=\sqrt{\frac{2}{\pi}}=0.7979$ and  $\sqrt{\V |Z|} =\frac{\sqrt{\pi-2}}{\sqrt{\pi}}=0.6028$.
\end{itemize}
The values $\gamma=0.8$, $0.9$, $0.95$ corresponding to the percentiles $1.28$, $1.645$, $1.96$ were chosen in order to evaluate especially the final part of the distribution which is more important for testing. The computation of the empirical distribution of $K_n^{sd}$ allow to appreciate either the precision of the asymptotic normalizing constant $d_n$ and the quality of the normal approximation.

The results in Tables \ref{tab:1} to \ref{tab:3} are quite illuminating and show that asymptotic normality (in absolute value) holds quite well for different cases of stable rv, different values of the long memory parameter and even for relatively small sample sizes $n$. This can be clearly appreciated by inspecting closely the results for $K_n^{sd}$. Inspection of the results for $K_n^{*}$ show that the asymptotic normalizing constant $d_n$ may not always be otpimal, especially if $D$ is small. The results show clear convergence to the theoretical values as sample size $n$ increases. In the case $D=0.8$ the $5$\% significant level test is quite precise, eventually a bit conservative, in all cases and for small sample sizes.

In practice one actually needs a $\log n$- consistent estimate  of the normalizing constant; one can consult \citet{DalGir06} and the references therein for $\log n$- consistent estimators of the relevant quantities. 


\end{document}